\documentclass{amsart}

\usepackage{url}
\usepackage[all]{xy}

\newtheorem{theorem}{Theorem}[section]
\newtheorem{proposition}[theorem]{Proposition}
\newtheorem{lemma}[theorem]{Lemma}
\newtheorem{corollary}[theorem]{Corollary}
\newtheorem{fact}[theorem]{Fact}

\theoremstyle{definition}
\newtheorem{definition}[theorem]{Definition}

\theoremstyle{remark}
\newtheorem{remark}[theorem]{Remark}

\numberwithin{equation}{section}

\def\Ind{\setbox0=\hbox{$x$}\kern\wd0\hbox to 0pt{\hss$\mid$\hss} \lower.9\ht0\hbox to 0pt{\hss$\smile$\hss}\kern\wd0} 

\def\Notind{\setbox0=\hbox{$x$}\kern\wd0\hbox to 0pt{\mathchardef \nn=12854\hss$\nn$\kern1.4\wd0\hss}\hbox to 0pt{\hss$\mid$\hss}\lower.9\ht0 \hbox to 0pt{\hss$\smile$\hss}\kern\wd0}

\def \d {\delta}

\def \U {\mathcal U}

\def \DCF {\operatorname{DCF}}
\def \B {\mathcal B}

\title[Differential Galois cohomology and PPV extensions]{Differential Galois cohomology and Parameterized Picard-Vessiot extensions}
\author{Omar Le\'on S\'anchez}
\address{Omar Le\'on S\'anchez\\
University of Manchester\\
Department of Mathematics\\
Oxford Road \\
Manchester, M13 9PL.}
\email{omar.sanchez@manchester.ac.uk}

\author{Anand Pillay}
\address{Anand Pillay\\
University of Notre Dame\\
Department of Mathematics\\
255 Hurley, Notre Dame\\
IN, 46556.}
\email{apillay@nd.edu}
\thanks{Anand Pillay was supported by NSF grants DMS-1665035 and  DMS-1760212, as well as a Kathleen Ollerenshaw Visiting Professorship at the University of Manchester}

\date{\today}

\begin{document}

\begin{abstract}
Assuming that the differential field $(K,\d)$ is differentially large, in the sense of Le\'on S\'anchez and Tressl \cite{LSTressl2018}, and ``bounded" as a field,  we prove that for any linear differential algebraic group $G$ over $K$, the differential Galois (or constrained) cohomology set $H^1_\d(K,G)$ is finite. This applies, among other things, to \emph{closed ordered differential fields} in the sense of Singer \cite{Singer1978}, and to \emph{closed $p$-adic differential fields} in the sense of Tressl \cite{Tressl}. As an application, we prove a general existence result for parameterized Picard-Vessiot extensions within certain families of fields; if $(K,\d_x,\d_t)$ is a field with two commuting derivations, and $\d_x Z = AZ$ is a parameterized linear differential equation over $K$, and $(K^{\d_x},\d_t)$ is ``differentially large" and $K^{\delta_x}$ is bounded, and $(K^{\d_x}, \d_t)$ is existentially closed in $(K,\d_t)$, then there is a PPV extension $(L,\d_x,\d_t)$ of $K$ for the equation such that $(K^{\delta_x},\d_t)$ is existentially closed in $(L,\d_t)$. For instance, it follows that if the $\d_x$-constants of a formally real differential field $(K,\d_x,\d_t)$ is a \emph{closed ordered $\d_t$-field}, then for any homogeneous linear $\d_x$-equation over $K$ there exists a PPV extension that is formally real. Similar observations apply to $p$-adic fields.
\end{abstract}

\maketitle

%\tableofcontents

\section{Introduction}
The paper has a dual aim; first to prove finiteness theorems for Kolchin's constrained cohomology  (or what we are now calling differential Galois cohomology), over suitable differential fields $(K,\d)$. Secondly to apply these finiteness theorems to the so-called Parameterized Picard-Vessiot (PPV) theory; namely to show the existence of PPV-extensions of differential fields $(K,\d_x, \d_t)$ for linear equations $\d_x Y = AY$ over $K$,  under conditions on the $\d_t$ field of $\d_x$-constants of $K$, and where we may demand more properties of the PPV-extension. 

The first is modelled on the {\em triviality} result from \cite{Pillay2017}, but requires some new ideas. The second is a relatively routine adaptation of  \cite{KP}, but building on and depending on \cite{LSN}, and with a few delicate points.

The constrained cohomology set $H^{1}_{\d}(K,G)$ (for $(K,\d)$ a differential field and $G$ a differential algebraic group over $K$), was defined in \cite{KolchinBook2}. It parameterizes differential algebraic principal homogeneous spaces $X$ for $G$ over $K$, up to isomorphism over $K$ in the appropriate category.

We will work in characteristic $0$. Recall that a field $K$ is called {\em bounded} if it has only finitely many algebraic extensions of any given degree.  A field $K$ is called {\em large} if whenever $V$ is a $K$-irreducible variety over $K$ with a smooth $K$-point then $V$ has a Zariski-dense set of $K$-points.  A rather strong notion of {\em differentially large} was given in \cite{LSTressl2018}:  $(K,\d)$ is differentially large if $K$ is large (as a field) and for any differential extension $(L,\d)$ of $(K,\d)$ if $K$ is existentially closed in $L$ as fields, then $(K,\d)$ is existentially closed in $(L,\d)$.  Other weaker notions of differential largeness will be discussed later. Many examples of differentially large fields come from models of Tressl's uniform model companion \cite{Tressl}; namely suppose that $T$ is a model-complete theory of large fields, then $T \cup\{\d$ is a derivation\} has a model companion, and any model will be differentially large.

As mentioned above, the basic template for our finiteness  proof comes from \cite{Pillay2017}, where {\em triviality} of the relevant cohomology sets is shown, assuming that $(K,\d)$ is algebraically closed and Picard-Vessiot closed.  We are weakening the algebraically closed assumption on $K$ to boundedness, and strengthening the PV-closed assumption to ``differentially large", and obtaining finiteness of the relevant cohomology sets. We will discuss possible refinements and  improvements of our  results later.

\section{Preliminaries}

We will here be working mainly  with (ordinary) differential fields $(K,\d)$ of characteristic $0$, although two commuting derivations will appear in Section 5.  $(\U,\d)$ denotes a saturated differentially closed field containing $(K,\d)$, and where needed $(K^{diff},\d)$ denotes a differential closure of $(K,\d)$.

We begin with a quick recollection on the $\tau$-functor, $\nabla$ map, D-varieties, $D$-groups, sharp points etc.   as well as on  differential Galois (or constrained) cohomology. 
(And again in Section 5 on applications to the PPV theory we will need to deal with relative or parameterized versions of these notions, as in \cite{LSN}.)

Many of these notions are due to Buium \cite{Buium} and have been discussed in many papers, such as \cite{Pillay-2remarks}, \cite{Kowalski-Pillay}, and \cite{LSTressl2018} with varying viewpoints. 

If $X \subseteq {\mathbb A}^{n}$ is an affine variety over $K$ whose ideal is generated by polynomials $P_{1}({\bar x}),...., P_r({\bar x})$ then  the prolongation $\tau(X)\subseteq {\mathbb A}^{2n}$ of $X$  is the variety over $K$ defined by  $P_{i}({\bar x}) = 0$ for $i=1,..,r$ as well as   $\sum_{j=1,..,n}( \partial P_{i}/\partial x_{j})u_{j} + P_{i}^{\partial} = 0$  for $i=1,..,r$ where $P_{i}^{\d}$ is the polynomial in ${\bar x}$ obtained from $P_{i}$ by applying $\d$ to its coefficients.  If $X$ is an abstract variety over $K$ then working locally in the affine charts, gives a well-defined  variety $\tau(X)$.
Alternatively one can use Weil descent as in Definition 2.7 of \cite{LSTressl2018}.

When $X$ is defined over the constants of $K$ then $\tau(X)$ coincides with the tangent bundle $T(X)$ of $X$. In general $\tau(X)$ is a {\em torsor} for $T(X)$: each fibre $\tau(X)_{a}$ is a principal homogeneous space for the tangent space $T(X)_{a}$ at $a\in X$  (and uniformly so). 

When $X$ is  affine, and $a$ is a point in $X(\U)$  then $\d(a)$ is the tuple obtained by applying $\d$ to the coordinates of $a$ and one sees readily that $(a,\d(a))\in \tau(X)(\U)$.  We call this map $\nabla_{X}:X\to \tau(X)$. It makes sense for any algebraic variety $X$ over $K$, and is of course a map  only  in the sense of {\em differential algebraic geometry}.

An important observation is that $\tau$ is a (covariant) functor that commutes with products and preserves fields of definition, from which one concludes:

\begin{remark} If $G$ is an algebraic group over $K$, then $\tau(G)$ has a natural structure of algebraic group over $K$, and the natural projection $\pi: \tau(G)\to G$ is a homomorphism. Moreover $\nabla_{G}:G\to \tau(G)$ is a homomorphic section for $\pi$. 
\end{remark}

If $G$ is an algebraic group over $K$, then by an (algebraic) torsor for $G$ over $K$, also called an (algebraic) PHS for $G$ over $K$, we mean an algebraic variety $X$ over $K$ together with a regular (strictly transitive) action of $G(\U)$ on $X(\U)$ over $K$ (in the sense of algebraic geometry). 

%Two such torsors $X,Y$ are said to be isomorphic over $K$ if there is an isomorphism of $X,Y$ over $K$ which commutes with the actions of $G$. The Galois cohomology pointed set $H^{1}(K,G)$ classifies these torsors, up to isomorphism over $K$. 

\begin{remark} Suppose $X$ is an algebraic torsor for the algebraic group $G$ over $K$. Then
\begin{enumerate}
\item $\tau(X)$ is an algebraic torsor for $\tau(G)$ over $K$.
\item  For $g\in G(\U)$, $x\in X(\U)$, $\nabla_{X}(g\cdot x) = \nabla_{G}(g)\cdot \nabla_{X}(x)$. 
\end{enumerate}
\end{remark}

Now for the notions of {\em algebraic D-variety} and {\em algebraic D-group}.
 An algebraic D-variety over $K$ is an algebraic variety $X$  equipped with a regular (in the sense of algebraic geometry) section $s:X\to \tau(X)$ of the projection $\pi:\tau(X) \to X$, all over $K$.  The data $s$ is equivalent to a lifting of the derivation $\d$ on $K$ to a derivation of the structure sheaf of $X$. 
An algebraic D-group over $K$ is an algebraic D-variety $(G,s)$ over $K$ where $G$ is an algebraic group and $s:G\to\tau(G)$ a homomorphism.  There is a natural category of D-varieties. In particular a D-subvariety of a D-variety $(X,s)$ is a subvariety $Y$ such that $s|Y: Y\to \tau(Y)$. 

\begin{remark}\label{factsDvar} \
 Suppose $(V,s)$ is an algebraic D-variety over $K$. If $W$ is a $K$-irreducible component of $V$, then $W$ is a D-subvariety of  $V$
\end{remark}

Let $(G,s_G)$ be an algebraic D-group over $K$. By a D-torsor for $(G,s_G)$ over $K$ we mean a D-variety $(X,s_X)$ over $K$ such that $X$ is an algebraic torsor for $G$ and for all $g\in G$ and $x\in X$ we have
$$s_X(g\cdot x)=s_G(g)\cdot s_X(x)$$
where the latter action is the (regular) action of $\tau(G)$ on $\tau(X)$  (mentioned earlier).

The analogues of algebraic varieties for differential algebraic geometry are {\em differential algebraic varieties}, which are modelled  locally on solution sets of systems of differential polynomial equations in $\U$ (see \cite{KolchinBook2}, \cite{Buium}). In place of  the Zariski topology we have the Kolchin topology, where on affine $n$-space for example, the Kolchin closed sets are the common zero sets of finite systems of differential polynomial equations. In this paper we will be only concerned with differential algebraic groups and their differential algebraic torsors, which coincide with groups and their torsors in the definable category, i.e. in the structure  $(\U,+,\times,\d)$.  For example, any definable group (over $K$) has the structure of a differential algebraic group over $K$, giving an equivalence of categories. An exhaustive account of all of this appears in \cite{Pillay-foundational}, which also gives an introduction to model theory. We will be concerned mainly with linear differential algebraic groups over $K$, namely  differential algebraic subgroups of some $GL_{n}$, defined over $K$ (or rather with an embedding, definable over $K$ into some $GL(n,\U)$). 

In any case we have the obvious notion of a differential algebraic torsor $X$ for a differential algebraic group, all over $K$.

There is a close connection between ``finite-dimensional"  (finite Morley rank) differential algebraic varieties and algebraic D-varieties via the $\sharp$-points functor.  Let $(X,s)$ be an algebraic D-variety over $K$. Then $(X,s)^{\sharp}$  (also called  $(X,s)^{\d}$) is $\{a\in X(\U): s(a) = \nabla(a)\}$, and is a finite-dimensional differential algebraic variety, and any finite-dimensional differential algebraic variety essentially arises this way (see \cite{Pillay-2remarks}).   For  finite-dimensional differential algebraic groups and their differential algebraic torsors we have a closer relationship.

\begin{remark}\label{equivcat} 
\begin{enumerate}
\item There is an equivalence of categories between algebraic D-groups over $K$ and differential algebraic groups of finite Morley rank over $K$. The functor is given by taking sharp points.
\item More generally, consider the following categories. Let $\mathcal C$ be the category of algebraic D-torsors for algebraic D-groups over $K$ with morphisms between $X^G$ and $Y^G$ being D-morphisms over $K$ that preserve the $G$-actions (i.e., $G$-morphisms). On the other hand, let $\mathcal D$ be the category of differential torsors for differential algebraic groups of finite Morley rank over $K$ with morphisms being differential morphisms over $K$ preserving the action. The $\sharp$-points functor yields an equivalence of categories. 
\end{enumerate}
\end{remark}

For arbitrary, not necessarily finite-dimensional differential algebraic torsors, we have (see \cite{Pillay-foundational}):

\begin{remark} Suppose $X$ is a differential  algebraic torsor of a differential algebraic group $G$, over $K$.
Then there is a definable over $K$ embedding of $(G, X)$ into some $(G_{1},X_{1})$ where $G_{1}$ is an algebraic group over $K$, $X_{1}$ an algebraic torsor for $G_{1}$ over $K$ and the action of $G_{1}$ on $X_{1}$ restricts to the action of $G$ on $X$. 
\end{remark}

Now for (differential) Galois cohomology.  The usual Galois cohomology (pointed) set $H^{1}(K,G)$ classifies algebraic torsors $X$ for $G$ over $K$ up to isomorphism over $K$ (as $G$-torsors). Here  $K$ is a field and $G$ is an algebraic group over $K$.   The generalization to differential algebraic groups over $K$ was carried out by Kolchin (in Chapter VII of \cite{KolchinBook2}).  $G$ is now a differential algebraic group over a differential field $K$ and the set of differential algebraic torsors for $G$ over $K$, up to differential algebraic over $K$ isomorphism can be described in terms of certain 1-cocycles from $Aut(K^{diff}/K)$ to $G(K^{diff})$. Hence the expression ``constrained cohomology",  as Kolchin referred to $K^{diff}$ as the constrained 
closure of $K$.  The set-up was generalized  in \cite{Pillay-Galois} to a model-theoretic setting.  In any case $H^{1}_{\d}(K,G)$ classifies the differential algebraic PHS's over $K$ for $G$.

As a matter of notation the expression ``differential Galois cohomology" was used in Kolchin's first book \cite{KolchinBook1} for a distinct, but related, notion (related to the Galois theory of strongly normal extensions).  The various cohomology theories of Kolchin and their interrelations, are discussed in \cite{Chatzidakis-Pillay}. Anyway we hope that our  identification of   ``constrained cohomology" and ``differential Galois cohomology",  is acceptable. 

The following lemma yields, in particular, an \emph{inductive principle} in cohomology sequences (similar to Lemma 2.1 of \cite{Pillay2017}) which is one of the key points in the proof of Theorem~\ref{mainfinite} below. The proof of the lemma makes use of results of Kolchin appearing in Chapter VII of \cite{KolchinBook2} and also arguments from \cite{Pillay-Galois} and \cite{PR}. We recall that given a differential algebraic group $N$ over $K$, by a $K$-form of $N$ one means a differential algebraic group over $K$ which is definably isomorphic to $N$ over $K^{diff}$. One way to build $K$-forms is as follows: Suppose $N$ is differential algebraic subgroup of a differential algebraic group $G$, both over $K$. Let $X$ be a differential algebraic torsor for $G$ over $K$, by \cite{Pillay-Galois} $X$ can be identified with a (constrained) definable 1-cocycle $\mu$ over $K$ in $G$ (so the cohomology class of $\mu$ is an element of $H^{1}_\d (K,G)$). Consider the map $\Phi$ from $Aut(K^{diff}/K)$ to the group of $K^{diff}$-definable group automorphisms of $N$ given by taking any $\sigma$ to conjugation of $N$ by $\mu(\sigma)$. This map $\Phi$ is also a definable cocycle in the sense of \S4 of \cite{Pillay-Galois} and furthermore, by an adaptation of the arguments in that section, $\Phi$ gives rise to a $K$-form of $N$ that we denote by $N_\mu$. We note that if $\mu$ is a definable 1-coboundary over $K$ in $G$ (namely $\mu$ is cohomologous to the trivial cocycle) then $N_\mu$ is isomorphic to $N$ over $K$. 
 
\begin{lemma} \label{inductive}
\begin{enumerate}
\item Let $G$ be an algebraic group over the differential field $K$. Then $H^{1}(K,G) = H^{1}_{\d}(K,G)$.
\item Let  $1\to N \to G\to H\to 1$ be a normal short exact sequence of differential algebraic groups over $K$.
Then the sequence  
$$H^{1}_{\d}(K,N) \to H^{1}_{\d}(K,G) \to H^{1}_{\d}(K,H)$$ 
of pointed sets is exact. Furthermore, if $H^{1}_{\d}(K,H)$ is finite and for every $\mu\in H^{1}_{\d}(K,G)$ the cohomology $H^{1}_{\d}(K,N_\mu)$ is also finite, then  $H^{1}_{\d}(K,G)$ is finite as well. 
\end{enumerate}
\end{lemma}
\begin{proof}
(1) is Theorem 4 from Chapter VII of Kolchin's book \cite{KolchinBook2}.

\vspace{2mm}
\noindent
(2) By Theorem 2 in Chapter VII of Kolchin's book \cite{KolchinBook2}, the short exact sequence of groups gives rise to the exact sequence (of pointed sets) in cohomology. We observe that when the group $G$ is abelian the sequence in cohomology is an exact sequence of groups and the ``furthermore clause'' follows immediately in this case. However, in general ($G$ not necessarily abelian) the sequence is only an exact sequence of pointed sets and the finiteness of $H^{1}_{\d}(K,G)$ in the ``furthermore clause'' needs to be justified~\footnote{We thank an anonymous referee for pointing out this subtle issue}. We remark that the analogue of (2) is well-known for the usual Galois cohomology of fields and algebraic groups; see for instance Theorem 6.16 of Platonov\&Rapinchuk's book \cite{PR}. The proof in the algebraic case adapts to the differential algebraic environment. Here we simply point out where the adaptations are needed and leave details to the interested reader. So we are now assuming that $H^{1}_{\d}(K,H)$ is finite and for every $\mu\in H^{1}_{\d}(K,G)$ the cohomology $H^{1}_{\d}(K,N_\mu)$ is also finite. As remarked above we already have $\pi:H^{1}_{\delta}(K,G) \to H^{1}_{\delta}(K, H)$. Hence the image $\pi(H^{1}_{\delta}(K,G))$ is finite, as it is contained in $H^{1}_{\delta}(K,H)$. So there are $
\mu_{1},\dots,\mu_{r}\in H^{1}_{\delta}(K,G)$ with $\pi(H^{1}_{\delta}(K,G))= \pi(\{\mu_{1},\dots,\mu_{r}\})$. Now the key point, which comes from adapting the arguments in \S1.3.2 of Platonov\&Rapinchuk book \cite{PR}, is that, for each $i=1,\dots,r$, $H^{1}_{\delta}(K,N_{\mu_i})$ maps onto $\pi^{-1}(\pi(\mu_{i}))$. Thus, by our hypothesis, each $\pi^{-1}(\pi(\mu_{i}))$ is finite, and so $H^{1}_{\delta}(K,G)$ is finite as well.

\end{proof}

A basic theorem of Serre \cite{Serrebook} is that if the field $K$ is bounded (finitely many extensions of degree $n$ for any $n$), then $H^{1}(K,G)$ is finite for any linear algebraic group $G$ over $K$. A motivating theme of this paper and future work is to generalize Serre's theorem in suitable ways  to finiteness theorems for the differential Galois cohomology of linear differential algebraic groups.

\smallskip

Finally let us discuss ``differential largeness" more.  As mentioned earlier a field $K$ is called large if any $K$-irreducible variety over $K$ with a smooth $K$-point has a Zariski-dense set of $K$-points. One of the points of largeness of $K$ is that the condition that a variety over $K$ has a dense set of $K$-points becomes first order, in definable families. 
Largeness (of the field of constants of a given differential field) played a role in the strong existence theorems for strongly normal extensions in \cite{KP} and \cite{BCPP}.   Analogous notions of ``differential largeness" of a differential field should give denseness, in the sense of the Kolchin topology,  of the set of $K$-points of a differential algebraic variety $X$ over $K$, under appropriate assumptions. The question is what the assumptions should be. One natural idea is to define the notion of a ``smooth point" on the differential algebraic variety $X$, and take as our assumption that $X$ has a smooth $K$-point. This idea was pursued in \cite{Cousins-thesis}.  Another idea is to take rather a weaker assumption that certain {\em algebraic} varieties attached to $X$ have many $K$-points. This was pursued by the first author and Marcus Tressl in \cite{LSTressl2018} and is the notion of differential largeness used in the current paper. 
A convenient definition from \cite{LSTressl2018}  (also appearing in \cite{Cousins-thesis}) is:

\begin{definition} Let $(K,\d)$ be a differential field. We call $K$ differentially large, if $K$ is large as a field, and for any differential field extension $(L,\d)$ of $(K,\d)$, if $K$ is existentially closed in $L$ as a field, then it is existentially closed in $L$ as a differential field.
\end{definition}

An equivalent condition (see Corollary 6.5 of \cite{LSTressl2018}) is a slight variant on the geometric axioms for $DCF_{0}$:

\begin{fact} $(K,\d)$ is differentially large iff $K$ is large as a field, and whenever $(V,s)$ is an algebraic D-variety over $K$ such that $V$ is $K$-irreducible, and has a smooth $K$-point, then for any nonempty Zariski open subset $U$ of $V$, over $K$, there is a $K$-point $a$ in $U$ such that $s(a) = \nabla_{V}(a)$. 
\end{fact}

\begin{remark}
\
\begin{enumerate}
\item  The conclusion on the right hand side of Fact 2.8  says precisely that $(V,s)^{\sharp}(K)$ is Kolchin dense in $(V,s)^{\sharp}$.

\item  Assuming that $T$ is a model-complete theory of large fields  (in the ring language), then $T\cup\{\d$ is a derivation\} has a model companion which is axiomatized by $T \cup \{\d$ is a derivation\} $\cup$ the axioms scheme given by the right hand side of Fact 2.8. This is Tressl's ``Uniform Model Companion" \cite{Tressl},  and provides a large source of differentially large, but not differentially closed fields, such as closed ordered differential fields, and closed $p$-adic differential fields. 

\item As pointed out in \cite{LSTressl2018}, algebraically closed and differentially large implies differentially closed. 

\end{enumerate}

\end{remark}

As in \cite{Pillay2017} we will make use of model-theoretic dimensions, Morley rank and $U$-rank, in the context of $DCF_{0}$.  (And as mentioned in \cite{Pillay2017} our arguments could be replaced by ``purely differential algebraic" arguments using Cassidy and Singer \cite{Cassidy-Singer-Holder}.)   We refer to \cite{Poizat}  for facts, background, and references about definable groups and dimensions in stable theories.   We have already mentioned ``finite rank" or ``finite-dimensional" differential algebraic groups a few times.  Finite rank can be taken to be finite Morley rank or finite $U$-rank and coincides with finite-dimensional in the sense of \cite{Buium} or ``differential type $0$" in the sense of \cite{Cassidy-Singer-Holder}. 
 In $DCF_{0}$,  the Morley rank of $G$ coincides with the $U$-rank of $G$ and is of the form $\omega\cdot m + k$, for $m, k$ nonnegative integers.  $G$ is said to be $1$-connected if it has no proper, nontrivial,  normal definable subgroup $N$ such that $G/N$ has finite rank (this coincides with the notion of strongly connected for differential algebraic groups \cite{Cassidy-Singer-Holder}).  From \cite[Corollary 6.3 and Theorem 6.7]{Poizat}$, G$ being $1$-connected is equivalent to $U(G) = \omega\cdot m$  (some $m\geq 0$) and $G$ being connected (no definable subgroup of finite index).  For any $G$, there is a maximal $1$-connected definable subgroup $N$ of $G$ (necessarily normal) such that $G/N$ has finite rank.  $N$ is called the $1$-connected component of $G$. 

In the bulk of this paper we are concerned with linear differential groups, that is differential algebraic (equivalently definable in $(\U,\d)$) subgroups of some $GL(n,\U)$. We will  be using repeatedly the fact \cite{Cassidy} that if $G$ is a linear differential algebraic group over $K$ and $N$ is a normal definable subgroup of $G$ defined over $K$, then $G/N$ is also linear over $K$, specifically $G/N$ is definably embeddable (over $K$) in some $GL(n,\U)$. 

Let us remark briefly that an arbitrary differential algebraic group $G$ defined over $K$, definably over $K$ embeds in an algebraic group $H$. And we conclude from this that $G$ has a normal definable subgroup $N$ such that $N$ is linear and $G/N$ embeds in an abelian variety.  (See \cite{Pillay-foundational}).

\vspace{5mm}
\noindent
A crucial technical result will be proved in  Section 3.
In Section 4 the finiteness theorem is proved (Theorem 4.1). In Section 5 we state and discuss the PPV existence theorem, and sketch the proof.

\section{The D-variety structure on $G$-maps}\label{modeltheory}

From now on we work inside a large saturated model $(\U,\d)\models \DCF_0$. Fix a differential field $(K,\d)$ of characteristic zero. Let $(G,s_G)$ be an algebraic D-group over $K$, and let $(X,s_X)$ and $(Y,s_Y)$ be $D$-torsors for $G$ also over $K$. 

We set
$$\mathcal B(G;X\times Y)=\{f:X\to Y: f \text{ is a $G$-isomorphism}\}.$$
When the context is clear we simply write $\B$ in place of $\B(G;X\times Y)$. Recall that a $G$-isomorphism $f$ is an isomorphism between $X$ and $Y$ (as algebraic varieties) such that $f(g\cdot x)=g\cdot f(x)$ for all $g\in G$, $x\in X$ and $y\in Y$; i.e., $f$ preserves the $G$-action. In this section we show that $\mathcal B$ has naturally the structure of a $D$-variety over $K$, and that its sharp $K$-points correspond to $G$-isomorphisms defined over $K$ that are $D$-morphisms.

First note that a $G$-isomorphism $f:X\to Y$ is completely determined by what it does to a single point of $X$; that is, fixing $a\in X$ once we know $f(a)$ then $f(x)$ must equal $g\cdot f(a)$ where $g\in G$ is such that $x=g\cdot a$. Thus, to any pair $(a,b)\in X\times Y$ we can associate a $G$-isomorphism $f:X\to Y$ given by 
$$f(x)=g\cdot b$$
where $g\in G$ is the unique element of $G$ with $x=g\cdot a$. Also, note that the graph of such an $f$ is simply given by the orbit of $(a,b)$ under the natural action of $G$ on $X\times Y$; namely, $g\cdot (x,y)=(g\cdot x,g\cdot y)$. In other words, each orbit determines uniquely a $G$-isomorphism and distinct orbits yield distinct maps. We thus have that the set $\mathcal B$ is in bijection, and we identify it, with the algebraic variety
$$(X\times Y)/G$$
which is defined over $K$.

Moreover, the canonical $D$-structure on $X\times Y$, namely 
$$s_{X\times Y}:=(s_x,s_Y):X\times Y \to \tau (X\times Y),$$ 
induces a $D$-structure on $\mathcal B$. Indeed, given $f\in \mathcal B$, let $(a,b)\in X\times Y$ be any point such that $\pi(a,b)=f$ where $\pi$ is the canonical (surjective) morphism $X\times Y\to \mathcal B$. Note that this is equivalent to $b=f(a)$. Let 
$$s_\B(f):=\tau \pi(s_{X\times Y}(a,b))\in \tau \B$$
and one just needs to show that this map is independent of the choice of $(a,b)$. So let $(a',b')\in X\times Y$ be in the fibre of $\pi$ above $f$. Note that $\tau \B$ is given as 
$$\tau(X\times Y)/\tau G$$
where the action of $\tau G$ on $\tau(X\times Y)=\tau(X)\times \tau(Y)$ is coordinate-wise. Thus, all we need to show is that there is $u\in \tau G$ such that 
$$s_{X\times Y}(a,b)=u\cdot s_{X\times Y}(a',b').$$
Since $(a,b)$ and $(a',b')$ are in the same $\pi$-fibre, there is $g\in G$ such that $(a,b)=g\cdot (a',b')$. We then have
\begin{align*}
s_{X\times Y}(a,b) & =  (s_X(g\cdot a'),s_Y(g\cdot b'))  \\
& =  (s_G(g)\cdot s_X(a'),s_G(g)\cdot s_Y(b')) \\
& =  s_G(g)\cdot s_{X\times Y}(a',b'). \\
\end{align*}
This shows that $s_\B:\B\to \tau \B$  as defined above is well-defined.
Note that we actually showed  that there is a (necessarily unique) map $s_\B$ that makes the following diagram commute.
$$\xymatrix{
\tau (X\times Y) \ar[rr]^{\tau \pi}&&\tau \B\\
X\times Y \ar[u]^{s_{X\times Y}}\ar[rr]^{\pi}&&\B \ar[u]_{s_\B}
}$$

\begin{remark}\label{smooth}
Note that, as an algebraic variety, $\B$ becomes isomorphic to $Y$ after naming a point in $X$. Indeed, fixing $a\in X$, the morphism that assigns $f\in \B$ to $f(a)\in Y$ is an isomorphism. In particular, each $K$-irreducible component of $\B$ is a smooth algebraic variety. We will use this latter fact in Corollary \ref{tosharp} below.
\end{remark}

So far we have shown that the algebraic variety $\B$ (defined over $K$) has a $D$-variety structure. We now prove further properties of this structure. 

\begin{proposition}\label{Bproperties}
For $(\B,s_\B)$ as above, we have
\begin{itemize}
\item [(i)] $K$-points of $\B$ correspond to $G$-isomorphisms defined over $K$, and
\item [(ii)] Sharp points of $(\B,s_\B)$ correspond to $G$-isomorphisms that are $D$-morphisms.
\end{itemize} 
\end{proposition}
\begin{proof}
(i) One just has to note that $f\in \B$ is a $K$-point of $\B$ if and only if the preimage $\pi^{-1}(f)$ is defined over $K$ (this uses that $\pi$ is a morphism over $K$). Note that this preimage is precisely the graph of $f$. 

(ii) Suppose $f$ is a sharp point of $(\B,s_\B)$. To prove that $f$ is a D-morphism, it suffices to show that if $a$ is a sharp point of $(X,s_X)$ then $f(a)$ is a sharp point of $(Y,s_Y)$. We have that $s_\B(f)=\nabla_B(f)$. In other words, there is $u\in \tau G$ such that 
$$s_{X\times Y}(a,f(a))=u\cdot \nabla_{X\times Y}(a,f(a)).$$
This means that $s_X(a)=u\cdot \nabla_X(a)$ and $s_Y(f(a))=u\cdot \nabla_Y(f(a)).$ As $a$ is a sharp point of $(X,s_X)$, then $u=\nabla_G(e)$ the identity of $\tau G$. Hence, $s_Y(f(a))=\nabla_X(f(a))$; i.e., $f(a)$ is a sharp point of $(Y,s_Y)$.

Now assume that $f\in \B$ is a $D$-morphism. We argue that it must be a sharp point of $(\B,s_\B)$. Let $(a,f(a))\in X\times Y$ be such that $a$ is sharp point of $(X,s_X)$. Then, as $f$ is a D-morphism, $(a,f(a))$ is a sharp point of $(X\times Y, s_{X\times Y})$. It follows, by definition of $s_\B$, that $s_\B(f)=\nabla_\B(f)$, as desired.
\end{proof}

%We now recall the definition of \emph{differentially large} in the sense of \cite{LSTressl2018}. We say that $(K,\d)$ is differentially large if $K$ is large as a field and for any differential field extension $(L,\d)$ the following holds: if $K$ is existentially closed in $L$ as fields then $(K,\d)$ is existentially closed in $(L,\d)$ as differential fields. In \cite{LSTressl2018} several characterisations of differentially large are established. In particular, one can axiomatise this class of differential fields with first-order sentences in the language of differential rings. Examples of differentially large fields are \emph{differentially closed fields}, \emph{closed ordered differential fields}, and \emph{closed $p$-adic differential fields}. In general, any large field equipped with a derivation has a differentially large extension which is an elementary extension in the pure field sense.

%\begin{remark}
%There are other notions in the literature of 'differential largeness'.... Greg's thesis...
%\end{remark}

The following is an important consequence of Proposition ~\ref{Bproperties}.

\begin{corollary}\label{tosharp}
Assume $(K,\d)$ is differentially large. With $(G,s_{G})$, $(X,s_{X})$ and $(Y,s_{Y})$ as above, if there is a $G$-isomorphism $f:X\to Y$ defined over $K$, then there is a differential $G^\sharp$-isomorphism $g:X^\sharp\to Y^\sharp$ defined over $K$. 
\end{corollary}
\begin{proof}
Let $(\B,s_\B)$ be the $D$-variety over $K$ of $G$-isomorphisms between $X$ and $Y$ as above. As we are assuming that there is a $G$-isomorphism over $K$, Proposition~\ref{Bproperties}(i) implies that $\B$ has a $K$-point. Let $\B_1$ be the $K$-irreducible component of $\B$ that contains this $K$-point. Recall that by Remark \ref{factsDvar},  $\B_1$ is a D-subvariety of $\B$ and by Remark \ref{smooth} it is also a smooth algebraic variety. As $(K,\d)$ is differentially large, there is a sharp $K$-point of $(\B_1,s_{\B_1})$, and in particular also of $(\B,s_\B)$. Call it $f$.

By Proposition~\ref{Bproperties}, $f$ is a $G$-isomorphism over $K$ that is also a $D$-morphism between $X$ and $Y$. Hence, the restriction of $f$ to $X^\sharp$ yields the desired $G^\sharp$-isomorphism. 
\end{proof}

\section{Finiteness of $H^1_\delta(K,G)$}\label{finiteH1}

We now prove the first main result. 

\begin{theorem}\label{mainfinite}
Suppose $(K,\d)$ is differentially large and also that $K$ is bounded as a field. Then, for any linear differential algebraic group $G$ over $K$, the differential Galois cohomology set $H^1_\delta(K,G)$ is finite.
\end{theorem}
\begin{proof}
We split the proof into two cases; namely, when $G$ is of finite rank and the general case. 
\medskip

\noindent {\bf Case 1.  $G$ has finite Morley rank.}   

By Remark~\ref{equivcat}(1), we may assume that $G$ is of the form $H^\sharp$ for a linear algebraic $D$-group $(H,s_H)$ over $K$. We will prove that 
$$|H^1_\d(K,G)| \leq |H^1(K,H)|$$
The result follows from this inequality. Indeed, as $K$ is bounded as a field, the right-hand-side is finite (see \cite[Chapter III, \S4.3]{Serrebook}).

Let $X$ and $Y$ be differential torsors for $G$ over $K$. By Remark~\ref{equivcat}(2), we may assume $X$ and $Y$ are of the form $V^\sharp$ and $W^\sharp$, respectively, for $(V,s_V)$ and $(W,s_W)$ algebraic $D$-torsors for $H$ over $K$. Now, if there is a $H$-isomorphism over $K$ between $V$ and $W$, then, by Corollary~\ref{tosharp}, there is a differential $G$-isomorphism over $K$ from $X$ and $Y$. This establishes the desired inequality.

\bigskip

\noindent {\bf Case 2.  $G$ is arbitrary.} 

The proof in this case is similar in strategy to that in  \cite[\S 2]{Pillay2017}, although we must take additional care as $K$ is not assumed to be algebraically closed.  We will make heavy use use of Lemma~\ref{inductive}. 

% following \emph{inductive principle}, which is an immediate consequence of Theorem 2 from \cite[Chapter VII, \S2]{KolchinBook2}. Suppose 
%$$1\to N\to G\to H\to 1$$ 
%is an exact sequence of differential algebraic groups over $K$. If $H^1_\d(K,N)$ and $H^1_\d(K,H)$ are finite, then $H^1_\d(K,G)$ is finite. Let $\bar G$ be the Zariski closure of $G$. Then $\bar G$ is a linear algebraic group over $K$. Let $N$ be a maximal normal subgroup of $\bar G$

\medskip

By the inductive principle (Lemma~\ref{inductive}) and Case 1, we may assume that $G$ is 1-connected; namely, $G$ has no proper differential algebraic subgroup $N$ such that the homogeneous space $G/N$ is of finite rank.

Exactly as in \cite{Pillay2017}, let $N$ be the (unique) maximal normal solvable definable $1$-connected subgroup of $G$.  Then $H= G/N$ is almost semisimple in the sense that $H$ has no proper nontrivial normal commutative definable $1$-connected subgroup.  By uniqueness $N$ (so also $H$) is  defined over $K$. 
Moreover, as mentioned earlier $H$ is also linear.  By Lemma~\ref{inductive} it is enough to deal with the cases where $G = N$ and $G = H$.

\medskip
{\bf Case 2(a).  $G = N$.}

Let $\bar G$ be the Zariski closure of $G$ (inside the ambient $GL_{n}(\U)$). Then $\bar G$ is a connected solvable linear algebraic group defined over $K$.  By Theorems 2.3 and 2.4 of \cite{PR}, $\bar G$ is the semidirect product of connected algebraic subgroups $T$ and $U$ of $\bar G$, both defined over $K$, where $U$ is unipotent, and normal in $\bar G$, and $T$ is an algebraic torus (that is, $T$ is isomorphic,  over the algebraic closure of $K$, to some product of copies of the multiplicative group, as an algebraic group).  

$U\cap G$ is normal in $G$ and defined over $K$,  and $G/(U\cap G)$ is also defined over $K$ and embeds in $T$.   By Lemma~\ref{inductive} (2) it suffices to prove that each of $U\cap G$ and $G/(U\cap G)$ has finite $H^{1}_{\partial}(K,-)$. 

Now by \cite{PR} $U$ has a central series defined over $K$ where each quotient is isomorphic over $K$ to the additive group $(\U, +)$.  Intersecting with $G$, $U\cap G$ has a central series defined over $K$ where each quotient $U_{i}$ say, definably over $K$, embeds  in $(\U,+)$.  We may assume by Case 1, that each such quotient $U_{i}$  has infinite Morley rank,  so must be definably isomorphic over $K$ to $(\U,+)$.   But $H^{1}(K,\mathbb{G}_{a})$ is trivial, hence by Lemma~\ref{inductive}(1) so is $H^{1}_{\partial}(K, U_{i})$.  We have shown that $H^{1}_{\partial}(K,(U\cap G))$ is finite. 

Now we should deal with $T_{1} = G/(U\cap G)$. There are different approaches.  The first is as follows: $T$ itself, as a connected algebraic torus, over $K$, may not be isomorphic over $K$ to some power of $\mathbb{G}_{m}$. But it is isomorphic over $K$ to a finite product of connected $1$-dimensional algebraic tori. Such a connected $1$-dimensional algebraic torus has in $DCF_{0}$ Morley rank $\omega$ and all proper differential algebraic subgroups have finite Morley rank. So we can proceed, as in the previous paragraph, to show that  $H^{1}_{\partial}(K,T_{1})$ is finite, using the hypothesis that $K$ is bounded as a field, Serre's theorem, and Lemma~\ref{inductive}.

Alternatively, one can use the standard logarithmic derivative, $dlog$, to reduce to  the case of (commutative) unipotent groups, which is dealt with above. (And this method may be useful in other situations.)  We have seen that $T_{1}$ is definably embeddable over $K$ into $T$ which is a commutative subgroup of $GL_{n}(\U)$.  This logarithmic derivative takes a matrix $X\in GL_{n}(\U)$ to $X'\cdot X^{-1}$, and restricted to $T$,  takes $T$ into a commutative (differential algebraic) unipotent group, defined over $K$.  The kernel of $dlog|T_{1}$ has finite Morley rank, and the image is contained in a commutative unipotent algebraic group, all defined over $K$. Induction, together with the earlier methods yield that 
$H^{1}_{\partial}(K,T_{1})$ is finite, as required. 

\vspace{2mm}
\noindent
{\bf Case 2(b). $G = H$ (is almost semisimple).}

The centre $Z$ of $G$ is finite-dimensional, by almost simplicity of $G$, so by  Lemma~\ref{inductive} (2) and Case 1, we may quotient by $Z$ (i.e. assume $Z$ is trivial), and so $G$ is semisimple in the sense of \cite{Cassidy-semisimple}: no normal definable nontrivial connected commutative subgroup. In fact $G$ will also have trivial centre (but we will not use this). 

Let $\bar G$ be the Zariski closure of $G$. By \cite[Theorem 14]{Cassidy-semisimple}, $\bar G$ is a connected semisimple algebraic group, defined over $K$.  We will prove that $G = \bar G$. We are  in the situation of Theorem 15 of Cassidy \cite{Cassidy-semisimple}. This Theorem 15 of Cassidy depends on  an earlier paragraph from Cassidy's paper which cites a result, Theorem 27.5 of Humphreys \cite{Humphreys} on semisimple algebraic groups defined over a field $K$.   Cassidy is working with arbitrary $K$, but Humphreys is assuming $K$ to be algebraically closed. So in fact this Theorem 15 from \cite{Cassidy-semisimple}, to be correct,  should be stated under an assumption that $K$ is algebraically closed.  We conclude:
\newline
(i) there are normal, nontrivial, simple (in the sense of algebraic groups), pairwise commuting, algebraic subgroups $A_{1},...,A_{r}$ of $\bar G$  such that the induced morphism $\pi$ from $A_{1}\times ... \times A_{r}$ to $\bar G$ is surjective with finite kernel (and moreover any normal simple algebraic subgroup of $\bar G$ is among the $A_{i}$), and
\newline
(ii) let $G_{i}$ be the connected component of $G\cap A_{i}$ (in the sense of groups definable in $(\U,\d)$) for $i=1,..,r$. Then each $G_{i}$ is simple (as a differential algebraic group), and Zariski-dense in $A_{i}$, and of course the $G_{i}$ pairwise commute.  Moreover the natural homomorphism which we call also $\pi$ from $G_{1}\times ...\times G_{r}$ to $G$ is surjective with finite kernel. 

Both the $A_{i}$ (and so $G_{i}$) will in general be defined only over the algebraic closure of $K$. 

As $G$ is $1$-connected, it is easy to see that each $G_{i}$ is $1$-connected too, so has infinite Morley rank.  But $G_{i}$ is Zariski-dense in the simple algebraic group  $A_{i}$, so by Theorem 17 of \cite{Cassidy-semisimple} or Proposition 5.1 of \cite{Pillay-fields}, $G_{i} = A_{i}$ for $i=1,..,r$. It follows easily that $G = \bar G$. Namely $G$ is already an algebraic group over $K$.  Our assumption that $K$ is bounded as a field implies that $H^{1}(K,G)$ is finite (by Serre's theorem) so by Lemma~\ref{inductive} (1), $H^{1}_{\partial}(K, G)$ is also finite. 

\end{proof}

\begin{remark}
Some final comments. 
\begin{enumerate}
\item [(i)] We have focused here on  linear differential algebraic groups, but the results should extend to arbitrary differential algebraic groups $G$ over $K$ with the conclusion that $H^{1}_{\d}(K,G)$ is countable, rather than finite.

\item [(ii)] The differential largeness assumption in Theorem 4.1 is rather strong and we would like to replace it by weaker conditions, eventually finding the appropriate differential analogue of boundedness of a field.  A first approximation would be the condition that every system of linear differential equations over $K$ has a fundamental system of solutions in $K$. This should be tractable although the proofs in \cite{Pillay2017} will not go through directly (as $K$ is no longer assumed to be algebraically closed). The ``right" differential analogue of boundedness should be that for each $n$ there are only finitely many linear DE's (in vector form) over $K$ of the form $\d Y = AY$ where $Y$ is a $n\times 1$ vector of indeterminates, up to gauge transformation over $K$  (together with boundedness of $K$ as a field).  Deducing the finiteness theorem for the differential Galois cohomology of linear differential algebraic groups over $K$, using this (tentative) notion of differentially bounded, seems to be substantially more difficult.

\item [(iii)]Also note from the proofs in this section that the main case is when $G$ is finite-dimensional. Namely, assuming that $K$ is bounded as a field, it follows that if $H^{1}_{\d}(K,G)$ is finite for every finite rank linear differential algebraic group over $K$, then the same holds without the finite rank assumption.

\item [(iv)] The results in this section also go through for fields with several commuting derivations (differential largeness having been formulated in \cite{LSTressl2018} in the partial case).  But for the sake of exposition we focused on the ordinary case.
\end{enumerate}
\end{remark}

\section{Existence of parameterized Picard Vessiot (PPV) extensions}

In this section we present the  application of Theorem \ref{mainfinite} to the existence of PPV extensions, with prescribed properties (as described in the abstract). We now assume that $(K,\d_x,\d_t)$ is a differential field (of characteristic zero) with two commuting derivations.  We refer to  \cite{LS} and \cite{LSN} for some model theory of differentially closed fields with respect to several commuting derivations (in our case two).  $K^{\d_x}$ denotes the field of $\d_x$-constants of $K$.  We sometimes use $\Pi$ to refer to the set $\{\d_x, \d_t\}$ of (commuting) derivations. 

Consider a homogeneous linear $\d_x$-equation of order $n$ in matrix form:
\begin{equation}\label{lineareq}
\d_x (Z)=A Z
\end{equation}
where $Z$ is an indeterminate varying in $GL_n$ and $A$ is in $Mat_n(K)$.

A PPV extension of $K$ for the equation is precisely a $\Pi$-field $L$ extending $K$ which is generated, as a $\Pi$-field by a solution $Z$ of the equation \eqref{lineareq}, and such that $L^{\d_x} = K^{\d_x}$, where  as above, $K^{\d_x}$ denotes the field of $\d_x$-constants of $K$  (which note is a $\d_t$-field). 

With the above notation we show:
\begin{theorem} Suppose that $(K^{\d_x},\d_t)$ is existentially closed in $(K,\d_t)$ as $\d_t$-fields,  $K^{\d_x}$ is bounded as a field, and $(K^{\d_x},\d_t)$ is differentially large (in the sense of Definition 2.7). Then there is a PPV extension $(L,\d_x,\d_t)$  of $(K,\d_x,\d_t)$ for the equation \eqref{lineareq}  such that $(K^{\d_x},\d_t)$ is existentially closed in 
$(L,\d_t)$ as $\d_t$-fields. 
\end{theorem}

This is the PPV generalization of Theorem 1.5 of \cite{KP}.  
In \cite{KP} the context was simply a linear $DE$ 

\begin{equation}\label{linearordinary}
\d Z = AZ
\end{equation}
over a differential field $(K,\d)$ where the constants $K^{\d}$ are asssumed to be bounded, large, and existentially closed in the field $K$. And we produced in Theorem 1.5 of \cite{KP} a PV extension $(L,\d)$ of $K$ for the equation such that $K^{\d}$ is existentially closed in $L$ as fields.   We have chosen here  to summarize the steps involved in the proof of Theorem 1.5 of \cite{KP}, and then show how these adapt/generalize to the PPV case, partly for pedagogical reasons, relevant to future generalizations.  Although just going through the required steps in  the new situation with references to \cite{KP} may have been sufficient. In any case, until the end of Remark 5.2 we will be working in the case of one derivation.

\vspace{2mm}
\noindent
(I) (Interpretations.)  Let $\mathcal Y$ be the solution set of the equation \eqref{linearordinary} in the universal domain $(\U,\d)$. 
Consider the $2$-sorted structure $(\U^{\d},{\mathcal Y})$ equipped with all $K$-definable structure from $(\U,\d)$. 
Then in a suitable language $L_{\d,K}$, $(\U^{\d},{\mathcal Y})$ is proved in \cite{KP} to be an elementary substructure of $(\U, GL(n,\U))$, where the latter is a suitable reduct of the algebraically closed field $\U$.  More precisely the language $L_{\d,K}$ consists of symbols $R_{X}$ for D-subvarieties of $\U^{r} \times (GL(n,\U))^{s}$ over $K$, namely subvarieties $X$ of $\U^{r} \times (GL(n,\U))^{s}$ which are defined over $K$ and such that  $((\U^{\d})^{r} \times {\mathcal Y}^{s})\cap X$ is Zariski-dense in $X$. The interpretation of $R_{X}$ in $(\U, GL(n,\U))$ is the tautolological one, and its interpretation in $(\U^{\d},\mathcal Y)$ is $((\U^{\d})^{r} \times {\mathcal Y}^{s})\cap X$.  What is important is that the subsets of $(\U^{\d})^{r} \times {\mathcal Y}^{s}$ which are definable over $K$ in the differentially closed field $(\U,\d)$ are precisely the sets which are (quantifier-free) definable without parameters in the $L_{\d,K}$-structure  $(\U^{\d},{\mathcal Y})$. 

The map assigning to each relation symbol $R_{X}$ of $L_{\d,K}$ the formula over $K$ defining $X$ in the algebraically closed field $\U$ yields an interpretation of $Th(\U^{\d},\mathcal Y)_{L_{\d,K}}$ in $ACF_{K}$  ($ACF$ with constants for elements of $K$) and we call this interpretation $\omega$. 

\vspace{2mm}
\noindent
(II) (Galois group and PV extensions.)  This is a construction of a certain $K$-definable function $f$ from ${\mathcal Y}$ to some power of $\U^{\d}$ which, in a sense to be made precise later, classifies
Picard-Vessiot extensions of $K$ for \eqref{linearordinary}.  Let $Aut({\mathcal Y}/K(\U^{\d}))$ be the group of  permutations of ${\mathcal Y}$ induced by automorphisms of the differential field $(\U,\d)$ which fix pointwise both $K$ and the constants $\U^{\d}$.  Let $b\in {\mathcal Y}$. Then the map $\rho_{b}$ taking $\sigma$ to $\sigma(b)b^{-1}$ is an isomorphism between  $Aut({\mathcal Y}/K(\U^{\d}))$ and a differential algebraic subgroup $H^{+}$ of $GL(n,\U)$ which is defined over $K$. Moreover $\rho_{b}$ does not depend on the choice of $b$ so we just call it $\rho$.  Let ${\mathcal Y}/H^{+}$ denote the set of right cosets of $H^{+}$ in ${\mathcal Y}$, equivalently the set of orbits under the action of $H^{+}$ on ${\mathcal Y}$ by left multiplication in $GL(n,\U)$.  Then ${\mathcal Y}/H^{+}$ is in $K$-definable bijection with a subset $\mathcal O$ of some Cartesian power of $\U^{\d}$, definable over $K\d$. And $f:{\mathcal Y} \to {\mathcal O}$ is the function that we wanted to construct. So $f(b_{1}) = f(b_{2})$ iff $b_{1} = hb_{2}$ for some $h\in H^{+}$. 

Moreover $b$ generates over $K$ a Picard-Vessiot extension for \eqref{linearordinary}  iff  $f(b)$ is a tuple of constants belonging to  $K$, i.e. in $K^{\d}$. 

\vspace{2mm}
\noindent
(III)   (Galois groupoid.) Given the map $f:{\mathcal Y} \to {\mathcal O}$, and $a\in {\mathcal O}$, ${\mathcal Y}_{a}$ denotes the fibre $f^{-1}(a)$ which is precisely $H^{+}b$ for some/any $b$ in the fibre.  For $a_{1},a_{2}\in {\mathcal O}$, $H_{a_{1},a_{2}}$ denotes $\{b_{1}^{-1}b_{2}: b_{1}\in {\mathcal Y}_{a_{1}}, b_{2}\in {\mathcal Y}_{a_{2}}\}$.  And $H_{a_{1}}$ denotes $H_{a_{1},a_{1}}$.  Note that multiplication is meant always in the sense of $GL(n,\U)$.  For $a\in {\mathcal O}$, $H_{a}$ is an algebraic subgroup of $GL(n, \U^{\d})$, and is the ``usual" Galois group of the equation \eqref{linearordinary}, as an algebraic group in the constants.

The Galois groupoid ${\mathcal G}$ of \eqref{linearordinary} has as objects the set $\mathcal O$, and for each $a_{1}, a_{2}$ the set $Mor(a_{1},a_{2})$ of morphisms between $a_{1}$ and $a_{2}$ is precisely $H_{a_{1},a_{2}}$. Lemma 4.3 of \cite{KP} explains why this is a groupoid, and why and how it is (quantifier-free)  definable over $K^{\d}$ in the algebraically closed field $\U^{\d}$.

 If $C$ is a field of constants containing $K^{\d}$ then  ${\mathcal G}(C)$ denotes the groupoid with objects ${\mathcal O}(C)$ and for $a_{1}, a_{2}\in {\mathcal O}(C)$,  the morphisms between $a_{1}$ and $a_{2}$ consists of the set $Mor(a_{1},a_{2})(C)$. (Namely just take $C$ points of everything.) The main point  is Proposition 4.6 of \cite{KP} which states that the set of Picard-Vessiot extensions of $K$ for \eqref{linearordinary}  is in natural one-to-one correspondence with  the set of connected components of the groupoid ${\mathcal G}(K^{\d})$.

\vspace{2mm}
\noindent
(IV)  (Galois cohomology.) The connection with Galois cohomology is as follows: Using the construction in (III),  assume that  ${\mathcal G}(K^{\d})$ is nonempty which means that there is $a\in {\mathcal O}(K^{\d})$. So $H_{a}$ is a linear algebraic group, defined over $K^{\d}$. Then the proof of Corollary 5.3 of \cite{KP} yields an injection of the set of connected components of ${\mathcal G}(K^{\d})$ into $H^{1}(K^{\d},H_{a})$.  So boundedness of the field $K^{\d}$ implies, via Serre's theorem, that ${\mathcal G}(K^{\d})$ has finitely many connected components.

\vspace{2mm}
\noindent
(V) (End of proof.)

Lemma 4.7 of \cite{KP} gives more information about the interpretation $\omega$  constructed in (I).  Recall that for a set $Z$, say, $\emptyset$-definable in the $L_{K,\d}$-structure $(\U^{\d},{\mathcal Y})$, $\omega(Z)$ is the corresponding set $\emptyset$-definable in the $L_{K,\d}$-structure  $(\U,GL(n,\U))$. And this of course extends to sets definable with  parameters. 
For notation:  Following \cite{KP} we define ${\mathcal O}(\U)$ to be $\omega({\mathcal O})$.  Let $f:{\mathcal Y}\to {\mathcal O}$, be as given in (II), then $F$ denotes $\omega(f)$, and note that  $F$ is a function from $GL(n,\U)$ onto ${\mathcal O}(\U)$  (of course definable over $K$ in the algebraically closed field $\U$).  Let $X_{a}$ denote $F^{-1}(a)$ for $a\in {\mathcal O}(\U)$. We write $H$ for the Zariski closure of $H^{+}$ in $GL(n,\U)$. (In fact as $H^{+}$ is $\emptyset$-definable in the $eq$ of the $L_{K,\d}$-structure  $(\U^{d}, {\mathcal Y})$ we could also define $H$ to be $\omega(H^{+})$.) We also write ${\mathcal G}(\U)$ for $\omega({\mathcal G})$. And likewise we write $Mor(\U)$ for $\omega(Mor)$ where $Mor$ is the set of morphisms of ${\mathcal G}$. 
Let $h:{\mathcal Y}\times {\mathcal Y} \to GL(n, \U^{\d})$ be $h(x,y) = x^{-1}y$.  Then  Lemma 4.7 of \cite{KP} says:
\newline
(i) $\omega(h):GL(n,\U)\times GL(n,\U) \to GL(n,\U)$ is precisely $\omega(x,y) = x^{-1}y$.
\newline
(ii) The fibres $X_{a}$ for $a\in {\mathcal O}(\U)$ are precisely the right cosets of $H$ in $GL(n,\U)$, so ${\mathcal O}(\U)$ is the homogeneous space $GL(n,\U)/H$.
\newline
(iii)  If $a\in {\mathcal O}(K^{\d})$ then $X_{a}$ is $K$-irreducible, is the Zariski closure of  ${\mathcal Y}_{a}$ and the Picard-Vessiot extension of $K$ corresponding to $a$ and \eqref{linearordinary} is precisely the function field $K(X_{a})$ of $X_{a}$. 
\newline
(iv) For $a_{1}, a_{2}\in {\mathcal O}(\U)$, $Mor(\U)(a_{1},a_{2}) = \{b_{1}^{-1}b_{2}: b_{1}\in X_{a_{1}}, b_{2}\in X_{a_{2}}\}$. 

(This basically follows by transfer from $(\U^{d},{\mathcal Y})$ via $\omega$. )

\vspace{2mm}
\noindent
Now, assuming that $K^{\d}$ is existentially closed in $K$ (as a field) and that $K^{\d}$ is bounded and large,  Theorem 1.5 of \cite{KP} is proved by first showing, using Lemma 4.7:

\vspace{2mm}
\noindent
(v)  For some elementary extension (as fields) $K^{1}$ of $K^{\d}$  which contains $K$, there is $a\in {\mathcal O}(K^{\d})$ such that $X_{a}$ has a $K_{1}$-rational point. 

\vspace{2mm}
\noindent
Now assuming for simplicity that $X_{a}$ is also $K_{1}$-irreducible (otherwise replacing it by a $K_{1}$-irreducible component which has a $K_{1}$-point), $K_{1}$ is existentially closed in the function field $K_{1}(X_{a})$ (as fields). Hence  $K^{\d}$ is existentially closed in $K^{1}(X_{a})$ as fields.  But $K^{\d} \leq K(X_{a})\leq K_{1}(X_{a})$, whereby $K^{\d}$ is existentially closed in $K(X_{a})$.  By (iii) from (V) above, we have found a Picard-Vessiot extension $L$ of $K$ for \eqref{linearordinary} such that $K^{\d}$ is existentially closed in $L$ (as fields). 

%\smallskip
\vspace{1mm}
%\noindent
\begin{remark} 
\begin{enumerate}
\item[(i)] Crespo, Hajto and van der Put \cite{CHvdP} prove the conclusion of Theorem 1.5 of \cite{KP} (in the linear case) when $K^{\d}$ is real closed or $p$-adically closed (and $K^{\d}$ is existentially closed in $K$) but their proof works only assuming boundedness (and largeness) of $K^{\d}$, and goes through the Tannakian  formalism.
\newline
\item [(ii)] In fact Theorem 1.5 of \cite{KP} is stated in the more general context of a logarithmic differential equation $dlog_{G}(y) = a$ on a not necessarily linear algebraic group $G$ defined over the constants of $K$ and where $a\in LG(K)$. In this situation the Tannakian theory (which is a linear theory) is not available, and so the technology in (I) to (V) above was really needed.
On the other hand from boundedness of $K^{\d}$ we only have countability (rather than finiteness) of $H^{1}(K^{\d}, G)$ for arbitrary algebraic groups over $K^{\d}$, but this implies finiteness of ``definable" chunks of $H^{1}(K^{\d},G)$, which was enough to obtain finitely many connected components of the set of $K^{\d}$ points of the relevant groupoid. 
\newline
\item [(iii)]  The paper \cite{KP} also included a simple existence theorem for PV (or strongly normal extensions), Theorem 1.3, as well as a certain uniqueness theorem, Theorem 1.4.  The existence theorem said that assuming only that $K^{\d}$ is existentiallly closed in $K$ as fields then there is a Picard-Vessiot extension of $K$ for \eqref{linearordinary}.  This used just steps (I) and (II) above, and was proved as follows: Let $F: GL(n,\U) \to {\mathcal O}(\U)$ be $\omega(f)$. Then $F$ applied to the identity of $GL(n,\U)$ is in ${\mathcal O}(K)$. So by existential closedness, ${\mathcal O}(K^{\d})$ is nonempty, yielding our Picard-Vessiot extension.  In fact as discussed in more detail below, this part of \cite{KP} was already extended to the $PPV$ (in fact the more general parameterized strongly normal) theory by the first author and Nagloo \cite{LSN}. 
\newline
\item [(iv)]  In the linear/PV case of Theorem 1.5 of \cite{KP}, actually the largeness assumption on $K^{\d}$ is not needed. This was pointed out by the first author to the second author, and due to the fact that if $G$ is a connected linear algebraic group over a field $K$, then $G(K)$ is Zariski-dense in $G$. So if $X$ is a PHS over $K$ for $G$ with a $K$-rational point then $X(K)$ is Zariski-dense in $X$. This is not necessarily the case for arbitrary algebraic groups such as abelian varieties.  So the more general case of Theorem 1.5 of \cite{KP} as discussed in (ii) above {\em does} need the largeness assumption on $K^{\d}$ (in addition to boundedness). 
\end{enumerate}
\end{remark}

\vspace{5mm}
\noindent
We will now begin going through the adaptation of steps (I) - (V) above to the PPV situation so as to prove Theorem 5.1 above. 
In the new situation of the equation \eqref{lineareq}, there  is also a universal domain $(\U,\d_x, \d_t)$, a saturated model of $DCF_{0,2}$ the theory of differentially closed fields with two commuting derivations.  As mentioned earlier we let $\Pi = \{\d_x,\d_t\}$. 
$(K,\d_x, \d_t)$ is a differential subfield of $\U$ and  we are given the equation \eqref{lineareq} above. ${\mathcal Y}$ now denotes the solution of \eqref{lineareq} in $\U$, and note again
${\mathcal Y}\subseteq GL(n,\U)$. 
The role of the $\d$ constants of $K$ and of   the universal domain is now played by $K^{\d_x}$ and $\U^{\d_x}$.  In fact the first author and Nagloo already carried out the generalizations of Steps (I) and (II) in \cite{LSN}, leading to the theorem that if $(K^{\d_x},\d_t)$ is existentially closed in $(K,\d_t)$ as $\d_t$-fields, then there is a PPV extension of $K$ for \eqref{lineareq}.  (Actually \cite{LSN}  works in the more general context  of a commmuting set $\Pi$ of $m$ derivations and a partition of $\Pi$ into nonempty sets ${\mathcal D}$ and $\Delta$, as well as  with $GL_{n}$ replaced by an arbitrary $\Delta$-algebraic group $G$ defined over $K^{\mathcal D}$.)   

The generalization of the notion of an algebraic $D$-variety from Section 1, is  that of a {\em parameterized} or {\em relative} $D$-variety from \cite{LS} and \cite{LSN}.  We will follow the notation in \cite{LSN}. 
By a $\d_t$-variety (or differential $\d_t$-variety) over $K$   we mean a subset $V$ of $\U^{n}$ defined by a finite system of $\d_t$-polynomials in  indeterminates  $\bar y=(y_1,\dots,y_n)$ say. Then by  the {\em parameterized prolongation} 
$\tau_{\d_x}(V)$ of $V$ is meant the $\d_t$-subvariety of $\U^{2n}$ defined by 
the following (differential) equations in indeterminates $y_{1},..,y_{n}, u_{1},..,u_{n}$:  First the set of $f(\bar y) = 0$ for all $\d_t$-polynomials over $K$ vanishing on $V$. 
Secondly the set of  $\sum_{i\geq 0, j=1,\dots,n} (\partial f/\partial(\delta^i y_j))\delta^i u_j   +   f^{\d_x}(\bar y)=0$ 
 where $f$ ranges over the $\d_t$-polynomials above, and where $f^{\d_t}$ means the resulting of hitting the coefficients of $f$ with $\d_t$.    A {\em parameterized} $D$-variety over $K$ is a pair $(V,s)$ where $V$ is a $\d_t$-variety over $K$, and $s: V\to \tau_{\d_t}(V)$ is a $\d_t$-polynomial section of the projection $\tau_{\d_t}(V)\to V$ on the first $n$-coordinates. 
By $\nabla_{\d_x}$ (or $\nabla_{V,\d_x}$) we mean the map taking $(v_{1},..,v_{n})\in V$ to $(v_{1},..,v_{n},\d_x(v_1),...,\d_x(v_n))\in \tau_{\d_t}(V)$. 

$(V,s)^{\sharp}$ denotes  $\{ {\bar v}\in V: s({\bar v}) = \nabla({\bar v})\}$, a (differential) $\Pi$-algebraic subvariety of $V$ (defined over $K$, if $(V,s)$ is).  There is a natural notion of a parameterized $D$-subvariety of $(V,s)$.  Several key facts  (analogues of facts about algebraic $D$-varieties) are given in Section 2 of \cite{LSN}, including the following characterization: 
Let $(V,s)$ be a parameterized $D$-variety, and let $W$ be a $\d_t$-subvariety of $V$. Then $W$ is a parametrized $D$-subvariety of $(V,s)$ iff $W\cap (V,s)^{\sharp}$ is Kolchin-dense in $W$. 

\vspace{2mm}
\noindent
(I)' (Interpretations.)  \cite{LSN} establishes the interpretation of the theory of the  two sorted structure  $(\U^{\d_x},{\mathcal Y})$ equipped with all relations definable with parameters from $K$ in $(\U,\d_x,\d_t)$  in the theory $Th(\U,\d_t)$ with parameters for elements of $K$.  The details of the interpretation and of the language chosen are somewhat  delicate, and this is maybe the most tricky part of the generalization of the single derivation situation. In the case at hand we have, in analogy with (I), relation symbols $R_{X}$ for $\d_t$-algebraic subvarieties $X$ of  $\U^{r}\times GL(n,\U)^{t}$  defined over $K$ such that $X\cap ((\U^{\d_x})^{r}\times {\mathcal Y}^t)$ is Kolchin-dense in $X$.  We call this language $L_{\Pi,K}$. 
 So
\begin{lemma} (\cite{LSN}, Corollary 4.3)   Consider both $(\U^{\d_x},{\mathcal Y})$ and  $(\U, GL(n,\U))$ as  $L_{\Pi,K}$-structures, where $R_{X}$ is interpreted as $X$ itself in $(\U, GL(n,\U))$ and as $X\cap ((\U^{\d_x})^{r}\times {\mathcal Y}^t)$ in $(\U^{\d_x},{\mathcal Y})$.  Then $(\U^{\d_x},{\mathcal Y})$ is an elementary substructure of $(\U, GL(n,\U))$. 
\end{lemma}

As previously we call the interpretation $\omega$. 

\vspace{2mm}
\noindent
(II)' (Galois group and PPV extensions.)
 This is discussed in Section 5 of \cite{LSN}, see Lemma 5.2 and Proposition 5.3 there, and is a straightforward adaptation of (II).  $Aut({\mathcal Y}/K\langle\U^{\d_x}\rangle)$ denotes the group of permutations of ${\mathcal Y}$ induced by automorphisms of the ambient differential closed field $(\U, \d_x, \d_t)$ which fix $K$ and $\U^{\d_x}$ pointwise (and there are other descriptions).  Here $K\langle\U^{\d_x}\rangle$ denotes the $\{\d_x, \d_t\}$-field generated over $K$ by 
$\U^{\d_x}$.   

Claims 1 and 2 in the proof of  Proposition 5.3 of \cite{LSN} establish that for $\sigma\in Aut({\mathcal Y}/K\langle\U^{\d_x}\rangle)$  and $\alpha\in {\mathcal Y}$, $\sigma(\alpha)\alpha^{-1}$ (multiplication in the sense of $GL(n,\U)$) does not depend on the choice of $\alpha$, and that the map taking $\sigma$ to $\sigma(\alpha)\alpha^{-1}$ establishes an isomorphism between $Aut({\mathcal Y}/K\langle\U^{\d_x}\rangle)$ and a $K$-definable subgroup of $GL(n,\U)$ which we again call $H^{+}$, the intrinsic Galois group.  By elimination of imaginaries let ${\mathcal O}$ again be the set of orbits under the action of $H^{+}$ on ${\mathcal Y}$ by left multiplication, a $K$-definable set in  $(\U,\d_x,\d_t)$ which up to $K$-definable bijection can be assumed to be a subset of some Cartesian power of $\U^{\d_x}$.  So, just as before, we get a $K$-definable function $f:\mathcal Y\to \mathcal O$ such that $f(b_1)=f(b_2)$ iff $b_1=h b_2$ for some $h\in H^+$. The following is a mixture of Proposition 5.3 from \cite{LSN} and its proof.

\begin{fact} If $b\in {\mathcal Y}$ and $a= f(b)\in K^{\d_x}$ then $K\langle b\rangle$ (the $\Pi$-differential field generated by $K$ and $b$) is a PPV extension of $K$ for \eqref{lineareq}, and in fact the formula $y\in {\mathcal Y} \wedge f(y) = a$ isolates the type of $b$ over $K\langle\U^{\d_x}\rangle$.
\end{fact}

And Theorem 5.4 of \cite{LSN} says:
\begin{corollary} Suppose $(K^{\d_x},\d_t)$ is existentially closed in $(K,\d_t)$ as $\d_t$-fields. Then there is a PPV extension of $K$ for \eqref{lineareq}

\end{corollary}

Again this holds and is stated in \cite{LSN} for the more general situation of $\d_x$-log differential equations over $K$ on an algebraic group $G$ defined over $K^{\d_x}$.  This reproves and generalizes the main results  of \cite{GGO}. 

\vspace{2mm}
\noindent
(III)'  (Galois groupoid.)  This goes through identically to (III) except that now the Galois groupoid is (quantifier-free) definable  over $K^{\d_x}$ in the differentially closed $\d_t$-field $(\U, \d_t)$.   Here are some details. We feel free to identify $H^{+}$ acting by left multiplication on ${\mathcal Y}$ with $Aut({\mathcal Y}/K\langle\U^{\d_x}\rangle)$.  

We start with the obvious:
\begin{fact} ${\mathcal Y}$ is a left coset of $GL(n,\U^{\d_x})$ in $GL(n,\U)$, namely of the form $bGL(n,\U^{\d_x})$ for some/any $b\in {\mathcal Y}$.  In particular for  all $b_{1},b_{2}\in {\mathcal Y}$, $b_{1}^{-1}b_{2} \in GL(n,\U^{\d_x})$. 
\end{fact}

Now, for $a\in {\mathcal O}$, note that ${\mathcal Y}_{a} = f^{-1}(a)$ is precisely an orbit under left multiplication by $H^{+}$, i.e. of the form $H^{+}b$ for some/any $b\in {\mathcal Y}_{a}$. 

Exactly as in Definition 4.1 and Remark 4.2 of \cite{KP} we have:

(*)    For $a_{1}, a_{2}\in {\mathcal O}$, define $H_{a_{1},a_{2}} = \{b_{1}b_{2}^{-1}: b_{1}\in {\mathcal Y}_{a_{1}}, b_{2}\in 
{\mathcal Y}_{a_{2}}\}$, which as in Remark 4.2 of \cite{KP} equals $\{b_{1}^{-1}b: b\in {\mathcal Y}_{a_{2}}\} =
\{b^{-1}b_{2}: b\in {\mathcal Y}_{a_{1}}\}$ for any fixed $b_{1}\in {\mathcal Y}_{a_{1}}$ and $b_{2}\in {\mathcal Y}_{a_{2}}$.

 With this notation we have:
\begin{lemma} (i)  $H_{a_{1}, a_{2}}$ is a subset of $GL(n,\U^{\d_x})$ (uniformly) definable over $K^{\d_x}$, $a_{1}, a_{2}$.
\newline
(ii) For any $c\in H_{a_{1},a_{2}}$, right multiplication by $c$ gives a bijection between ${\mathcal Y}_{a_{1}}$ and ${\mathcal Y}_{a_{2}}$.
\newline
(iii) For $a\in {\mathcal O}$, $H_{a}$ (which is by definition $H_{a,a}$) is a $\d_t$-algebraic subgroup of $GL(n,\U^{\d_x})$, and is the ``usual" Galois group of \eqref{lineareq}.  For any $b\in {\mathcal Y}$, $bH_{a} = {\mathcal Y}_{a} = H^{+}b$, so in particular $bH_{a}b^{-1} = H^{+}$.
\newline
(iv) For $a\in {\mathcal O}$, and any $b\in {\mathcal Y}_{a}$, the map taking $\sigma$ to $b^{-1}\sigma(b)$ is an isomorphism of groups between $Aut({\mathcal Y}/K\langle\U^{\d_x}\rangle)$ and $H_{a}$. 
\newline
(v) $H_{a_{1},a_{2}}\cdot H_{a_{2},a_{3}} = H_{a_{1},a_{3}}$. In particular $H_{a_{1},a_{2}}$ is a right coset (left torsor) of $H_{a_{1}}$ and a left coset (right torsor) of $H_{a_{2}}$.

\end{lemma}
\begin{proof}
(i) is Fact 5.6.
\newline
(ii) follows from (*).
\newline
(iii)  $H_{a}$ is clearly closed under inverses and and by (*) the product of two elements in $H_{a}$ has the form $b_{1}b_{2}^{-1}b_{2}b_{3}^{-1} = b_{1}b_{3}^{-1} \in H_{a}$  (where $b_{1}, b_{2}, b_{3}\in H_{a}$). 
Now $H_{a}$ is a subgroup of $GL(n,\U^{\d_x})$ and is definable over $K$ in $(\U,\d_x,\d_t)$.  As $(\U^{\d_x},\d_t)$ with the induced structure from $(\U,\d_x,\d_t)$ is just a $\d_t$-differentially closed field, and the definable closure of $K^{\d_x}$ in 
$\U^{\d_x}$ is $K^{\d_x}$, we obtain  the first sentence of (iii).  The rest is a simple computation using the definitition of $H^{+}$ as well as Fact 5.4 (that all elements of ${\mathcal Y}_{a}$ have the same type over $K\langle\U^{\d_x}\rangle$). 
\newline
Both (iv) and (v) are routine (and well-known). 
\end{proof}

The Galois groupoid attached to \eqref{lineareq} has ${\mathcal O}$ as its set of objects, and for $a_{1}, a_{2}\in {\mathcal O}$, $Mor(a_{1},a_{2})$ is precisely $H_{a_{1}, a_{2}}$.  Composition of morphisms is just multiplication in $GL(n,\U^{\d_x})$.  We call this groupoid ${\mathcal G}$.  We see from Lemma 5.7 that:

\begin{lemma} ${\mathcal G}$ is (quantifier-free) definable over $K^{\d_x}$ in the $\d_t$-differentially closed field $(\U^{\d_x},\d_t)$.  It is moreover connected (namely for each $a_{1}, a_{2}\in {\mathcal O}$, $Mor(a_{1},a_{2}) \neq \emptyset$). 
\end{lemma}

Now for any $\d_t$-differential field $(L,\d_t)$ (e.g. differential subfield of $(\U^{\d_x}, \d_t)$) containing $K^{\d_x}$, we can, by Lemma 5.8, consider the groupoid ${\mathcal G}(L)$, the set of objects ${\mathcal O}(L)$ being the interpretation in $L$ of the quantifier-free over $K^{\d_x}$-formula which defines ${\mathcal O}$ in $(U^{\d_x}, \d_t)$, and likewise for $H_{a_{1},a_{2}}(L)$ for $a_{1}, a_{2}\in {\mathcal O}(L)$ (which now may be empty).  Nevertheless it is easy to see that  ${\mathcal G}(L)$ is a groupoid, although possibly has more than one connected component.

\begin{lemma} The set of PPV extensions of $K$ (up to isomorphism over $K$ as $\Pi$-fields) for the equation \eqref{lineareq} is parameterized by, or in natural one-one correspondence with, the set of connected components of the groupoid  ${\mathcal G}(K^{\d_x})$.
\end{lemma}
\begin{proof} Every PPV extension of $K$ for \eqref{lineareq} is generated over $K$ by some $b\in {\mathcal Y}$ and clearly $f(b)\in {\mathcal O}(K^{\d_x})$. And conversely if $b\in {\mathcal Y}$ and $f(b)\in {\mathcal O}(K^{\d_x})$ then by Fact 5.4,  $K\langle b\rangle$ is a PPV extension of $K$. 

Now suppose that $b_{1}, b_{2}\in {\mathcal Y}$, $f(b_{i})\in {\mathcal O}(K^{\d_x})$, and that $K\langle b_{1}\rangle$ and $K\langle b_{2}\rangle$ are isomorphic over $K$ as $\Pi$-fields.. After applying an automorphism of $(\U,\Pi)$ which fixes pointwise $K$  (so fixes pointwise ${\mathcal G}(K^{\d_x})$)  we may assume that $K\langle b_{1}\rangle = K\langle b_{2}\rangle = L$ say.
So working in $GL(n,\U)$, $b_{2}^{-1}b_{1}\in  GL(n,L^{\d_x}) = GL(n,K^{\d_x})$, whereby $Mor(a_{1},a_{2})$ in ${\mathcal G}(K^{\d_x})$ is nonempty. 

Conversely, suppose that $f(b_{i}) = a_{i}\in {\mathcal O}(K^{\d_x})$ for $i=1,2$ and $Mor(a_{1},a_{2})$ is nonempty in ${\mathcal G}(K^{\d_x})$.  So there exist $b_{i}' \in {\mathcal Y}_{a_{i}}$, such that $b_{2}'^{-1}b_{1}'\in GL(n,K^{\d_x})$.  Then $K\langle b_{1}'\rangle = K\langle b_{2}'\rangle$. 
But by Fact 5.4,  $b_{1}$ and $b_{1}'$ have the same type over $K$, whence there is an isomorphism over $K$ between $K\langle b_{1}\rangle$ and $K\langle b_{1}'\rangle$. Likewise there is an isomorphism over $K$ between $K\langle b_{2}\rangle$ and $K\langle b_{2}'\rangle$. Hence $K\langle b_{1}\rangle$ and $K\langle b_{2}\rangle$ are isomorphic over $K$. This concludes the proof.
\end{proof} 

\vspace{2mm}
\noindent
(IV') (Differential Galois cohomology.) We are in the above situation of \eqref{lineareq} and its Galois groupoid $\mathcal G$,  definable over $K^{\d_x}$ in $(\U^{\d_x}, \d_t)$.
\begin{lemma}  Assume that $(K^{\d_x},\d_t)$ is differentially large, and that $K^{\d_x}$ is bounded as a field. Then ${\mathcal G}(K^{\d_x})$ has finitely many connected components.
\end{lemma}
\begin{proof} We may assume that ${\mathcal O}(K^{\d_x}) \neq \emptyset$.  So fix $a\in {\mathcal O}(K^{\d_x})$, and consider $H_{a}$, a $\d_t$-algebraic subgroup of  $GL(n,\U^{\d_x})$ defined over $K^{\d_x}$. 
The Claim in Corollary 5.3 of \cite{KP} goes through with no change, showing that for $b,c\in {\mathcal O}(K^{\d_x})$, $Mor(a,b)$ and $Mor(a,c)$ are isomorphic over $K^{\d_x}$,  as left torsors for $H_{a}$ in the category of $\d_t$-algebraic varieties, if and only if $Mor(b,c)(K^{\d_x})$ is nonempty.  This gives an embedding of the set of connected components of ${\mathcal G}(K^{\d_x})$ into $H^{1}_{\d_t}(K^{\d_x}, H_a)$ (where $K^{\d_x}$ is considered as a $\d_t$-field, and $H_{a}$ as a $\d_t$-algebraic group over $K^{\d_x}$).  So we conclude by Theorem 4.1.

\end{proof}

\vspace{2mm}
\noindent
(V') (End of proof of Theorem 5.1.) We go back to the interpretation $\omega$ of $Th(\U^{\d_x}, {\mathcal Y})$ (with all induced structure from sets definable over $K$ in $(\U,\Pi)$),  in $Th(\U, \d_t)$, using the precise formalism in (I)'.  Remember this gave $(\U^{\d_x}, {\mathcal Y})$ as an elementary substructure of $(\U, GL(n,\U))$ in the common language we called $L_{\Pi,K}$. 

Let $f:{\mathcal Y} \to {\mathcal O}$ be as given in (II)'. As remarked earlier  ${\mathcal O}(\U) = \omega({\mathcal O})$.  Let $F = \omega(f)$ which is a map from $GL(n,\U)$ to ${\mathcal O}(\U)$, which we know to be $K$-definable in the differentially closed field $(\U, \d_t)$. 

Again we write $X_{a}$ for $F^{-1}(a)$, where $a\in {\mathcal O}(\U)$. 

Recall that $H^{+}$ (the intrinsic Galois group of \eqref{lineareq}) is a $\Pi$-definable subgroup of $GL(n,\U)$, and we now let $H$ denote the $\d_t$-Kolchin closure of $H^{+}$, namely the smallest $\d_t$-definable subgroup of $GL(n,\U)$ containing $H^{+}$ (which exists by $\omega$-stability of $Th(\U,\d_t)$). 

We write ${\mathcal G}(\U)$ for $\omega({\mathcal G})$ which makes sense functorially. We write $Mor(\U)$ for $\omega(Mor)$. And we record the analogous facts about the interpretation $\omega$, to  what was said earlier in (V). 

\begin{lemma}
(i) Let $h:{\mathcal Y}\times {\mathcal Y} \to GL(n,\U^{\d_x})$ be $h(x,y) = x^{-1}y$. Then it is the same for $\omega(h): GL(n,\U)\times GL(n,\U) \to GL(n,\U)$.
\newline
(ii) The family of $X_{a}$ for $a\in {\mathcal O}(\U)$ is  precisely the family of right cosets of $H$ in $GL(n,\U)$.  So ${\mathcal O}(\U)$ identifies with $G/H$.
\newline 
(iii)  Let $a\in {\mathcal O}(K^{\d_x})$. Then $X_{a}$ is $K$-irreducible as a $\d_t$-algebraic set defined over $K$. Moreover the PPV extension of $K$ corresponding to $a$ (given by Fact 5.4) is precisely the $\d_t$-function field of $X_{a}$ over $K$, $K\langle X_{a}\rangle$, namely the $\d_t$-subfield of $(\U,\d_t)$ generated over $K$ by a generic point of $X_{a}$. 
\newline
(iv)  For $a_{1}, a_{2}\in {\mathcal O}(\U)$, $Mor(\U)(a_{1},a_{2}) = \{b_{1}^{-1}b_{2}: b_{1}\in X_{a_{1}}, b_{2}\in X_{a_{2}}\}$.

\end{lemma}
\begin{proof} (i) ${\mathcal Y}$ is the set of sharp points of a parameterized $D$-variety structure $s$ on  $GL(n,\U)$ over $K$. Hence we can find $b_{1}, b_{2} \in {\mathcal Y}$ which are  
$\d_{t}$-generic and independent over $K$  in $GL(n,\U)$.  Then  $b_{1}^{-1}b_{2}\in GL(n,\U^{\d_x})$, and we see that if $Z$ is the graph of the map taking $(x,y)\in GL(n,\U)\times GL(n,\U)$ to 
$x^{-1}y\in GL(n,\U)$, then $Z\cap ({\mathcal Y}\times {\mathcal Y} \times GL(n,\U^{\d_x}))$ is $\d_t$-Kolchin dense in $Z$, which yields that $Z = \omega(graph(h))$ as required. 
\newline
(ii) Fix $b\in {\mathcal Y}$. Let $h_{b}:{\mathcal Y} \to GL(n,\U^{\d_x})$ be left multiplication by $b^{-1}$. Then by part (i), $\omega(h_{b})$ is left multiplication by $b^{-1}$ from $GL(n,
\U)$ to $GL(n,\U)$.   Let $f_{1}$ be the composition $f\circ h_{b}^{-1}: GL(n,\U^{\d_x}) \to {\mathcal O}$.  By Lemma 5.3, $F_{1} :=\omega(f_{1}) = F\circ \omega(h_{b}^{-1}):GL(n,\U) \to {\mathcal O}(\U)$. 
Let $a=f(b)\in {\mathcal O}(\U^{\d_x})$.  Then the fibres of $f_{1}$ are the right cosets of $H_{a}$. All of this being definable in the differentially closed field $(\U^{\d_x}, \d_t)$, the fibres of $F_{1}$ are the right cosets of $H_{a}(\U)$.  Hence the fibres of $F$ are the right cosets of $bH_{a}(\U)b^{-1}$. By Lemma 5.7 (iii), $bH_{a}b^{-1} = H^{+}$.  Taking $\d_t$- Kolchin closures, we obtain that $bH_{a}(\U)b^{-1} = H$, as required. 
\newline
(iii)  Let $a\in {\mathcal O}(K^{\d_x})$, and let $b\in {\mathcal Y}_{a}$.  From the proof of (ii), $X_{a} = F^{-1}(a) = Hb$ which is the $\d_t$-Kolchin closure of $H^{+}b = {\mathcal Y}_{a}$ (see Lemma 5.7 (iii)). So $X_{a}$ is the Kolchin closure of ${\mathcal Y}_{a}$. Now all elements of ${\mathcal Y}_{a}$ have the same type over $K$ (in $(\U,\Pi)$), see Fact 5.4, which implies that $X_{a}$ is $K$-irreducible, as a $\d_t$-algebraic variety (over $K$).  Hence $b$ is the generic point of $X_{a}$ over $K$ as a $\d_t$-variety, and the $\d_t$-function field of $X_{a}$ over $K$ is precisely $K\langle b\rangle_{\d_t}$ the $\d_{t}$-field generated by $K$ and $b$. As  $b$ is a solution of \eqref{lineareq}, we see that $K\langle b\rangle_{\d_t} = K\langle b\rangle_{\Pi} =K\langle b\rangle$ (with earlier notation). 
\newline
(iv)  In the $L_{\Pi,K}$-structure $(\U^{\d_x}, {\mathcal Y})$ the following holds: for all $a_{1},a_{2}\in {\mathcal O}$, $Mor(a_{1},a_{2}) = \{b_{1}^{-1}b_{2}: b_{1}\in f^{-1}(a_{1}), b_{2}\in f^{-1}(a_{2})\}$.  Applying the interpretation $\omega$, Fact 5.4, and parts (i) and (ii) gives the required conclusion. 

\end{proof}

\vspace{2mm}
\noindent
We now finish the proof of Theorem 5.1.  We are assuming that $(K^{\d_x}, \d_{t})$ is existentially closed in $(K,\d_t)$, that  $K^{\d_x}$ is bounded as a field, and that $(K^{\d_x}, \d_t)$ is differentially large (in the sense of \cite{LSTressl2018}).  First:

\begin{lemma} There is $a\in {\mathcal O}(K^{\d_x})$ such that the $\d_t$-variety $X_{a}$ has a point in some elementary extension $(K_{1},\d_t)$ of $(K^{\d_x}, \d_t)$ which contains (extends) $(K, \d_t)$.
\end{lemma}
\begin{proof} By Corollary 5.5, ${\mathcal O}(K^{\d_x})$ is nonempty, giving rise to a PPV extension of $K$, but we need more. 
Considering $F:GL(n,\U) \to {\mathcal O}(\U)$,  let $a' = F(e)$  where $e$ is the identity of $GL(n,\U)$,  So $a'\in {\mathcal O}(K)$.  By 5.11 (ii), $X_{a'} = H$. As $(K^{\d_x}, \d_t)$ is e.c. in $(K,\d_t)$ there is, on general grounds, an elementary extension $(K_{1},\d_t)$ of $(K^{\d_x},\d_t)$  which contains $(K,\d_t)$, and we may assume that $(K_{1},\d_t)$ is also a differential subfield of $(\U,\d_t)$.  Now ${\mathcal G}(K^{\d_x})$ has only finitely many connected components, by Lemma 5.10.  Hence, as $(K_{1},\d_t)$ is an elementary extension of $(K^{\d_x},\d_t)$, there is $a\in {\mathcal O}(K^{\d_x})$ such that $Mor(a,a')(K_{1})$ is nonempty. Now $X_{a'}(K_{1})$ is nonempty as it contains the identity of $H$. Hence $X_{a}(K_{1})$ is nonempty, finishing the proof of the lemma.
\end{proof}

Let $a$ and $K_{1}$ be as in Lemma 5.12.  $X_{a}$ is a right coset of $H$, which is $K$-irreducible (as a $\d_t$-algebraic variety over $K$).  Let $Z$ be a $K_{1}$-irreducible component of $X_{a}$ which has a $K_{1}$-point. Then $Z$ is a coset of  the connected component $H^{0}$ of the $\d_t$-algebraic group $H$. 
Consider the $\d_t$-function field $K_{1}\langle Z\rangle_{\d_t}$.  Now general model-theoretic considerations imply that this $K_{1}\langle Z\rangle_{\d_t}$ is the function field of a coset $C$ of a proalgebraic group over $K_{1}$ such that $C$ has a $K_{1}$-rational point. As $K^{\d_x}$ is large as a field, so is $K_{1}$, whereby $K_{1}$ is existentially closed in $K_{1}\langle Z\rangle_{\d_t}$ as fields.  As $(K^{\d_x}, \d_t)$ is differentially large, so is  $(K_{1},\d_t)$ which implies that $K_{1}$ is existentially closed in $K_{1}\langle Z\rangle_{\d_t}$, as $\d_t$-fields.  As $K^{\d_x}$ is an elementary substructure of $K_{1}$, as $\d_t$-fields, it follows that $K^{\d_x}$ is existentially closed in $K_{1}\langle Z\rangle_{\d_t}$ as $\d_t$-fields. 
Finally, as the embedding of $K$ in $K_{1}$ extends to an embedding of $K\langle X_{a}\rangle_{\d_t}$ in $K_{1}\langle Z\rangle_{\d_t}$ as $\d_t$-fields, it follows that $K^{\d_x}$ is existentially closed in   $K\langle X_{a}\rangle_{\d_t}$ as $\d_t$-fields. As we know from Lemma 5.11 (iii) that $K\langle X_{a}\rangle_{\d_t}$ is a PPV extension of $K$ for \eqref{lineareq}, we are finished with the proof of Theorem 5.1. 

\vspace{5mm}
\noindent
As mentioned in Remark 2.9 (2), Theorem 5.1 applies to the situation where $T$ is a model complete theory of bounded large fields (in the language of unitary rings), and $(K^{\d_x},\d_t)$ is a model  of the model companion of $T\cup\{\d_t$ is a derivation\} (in the language of differential unitary rings). When $T$ is the theory of real closed fields, this model companion coincides with Singer's theory $CODF$ of closed ordered differential fields \cite{Singer1978} after adding a symbol for the unique ordering. 

\begin{corollary} Suppose $(K,\d_x,\d_t)$ is a field with commuting derivations $\d_x, \d_t$. Suppose that $K$ is formally real, and that $(K^{\d_x},\d_t)$ is a ``closed ordered differential field" (i.e. a model of CODF). Then, for any parameterized linear DE $\d_x(Z) = AZ$ over $K$ (as in \eqref{lineareq}) there is a PPV extension $(L,\d_x,\d_t)$ for the equation such that $L$ is formally real. 
\end{corollary}
\begin{proof} 
As $(K,\d_t)$ is a formally real differential field and $(K^{\d_x},\d_t)$ is a model of CODF, we see that $(K^{\d_x},\d_t)$ is existentially closed in $(K,\d_t)$ as $\d_t$-fields.
%As $K$ is formally real, it is a model of $T_{\forall}$ where $T = RCF$. So $(K,\d_t)$ extends tro a model of $CODF$, where by $(K^{\d_x}, \d_t)$ is exustentially closed in $(K,\d_t)$.  
Apply Theorem 5.1 to find a PPV extension $(L,\d_x,\d_t)$ such that $(K^{\d_x},\d_t)$ is existentially closed in $(L,\d_t)$. It follows that $L$ must be formally real. 
\end{proof}

Another example is when the theory $T$ is the theory of $p$-adically closed fields of rank $d$. In this case the model companion of $T\cup \{\d_t$ is a derivation\} is the theory of closed $p$-adic differential fields of rank $d$ introduced by Tressl in \cite{Tressl}. Similar to the above corollary we have:

\begin{corollary} Suppose $(K,\d_x,\d_t)$ is a field with commuting derivations $\d_x, \d_t$. Suppose that $K$ is a $p$-adic field of rank $d$, and that $(K^{\d_x},\d_t)$ is a ``closed $p$-adic differential field of rank $d$". Then, for any parameterized linear DE $\d_x(Z) = AZ$ over $K$ (as in \eqref{lineareq}) there is a PPV extension $(L,\d_x,\d_t)$ for the equation such that $L$ is a $p$-adic field of rank $d$. 
\end{corollary}

These two corollaries are parameterized versions of the main result of \cite{CHvdP}.

\bibliographystyle{plain}

\end{document}